\documentclass[12pt]{amsart}

\usepackage{amsthm,amsfonts,amsmath,amssymb,latexsym,epsfig,mathrsfs,yfonts,marvosym}
\usepackage[usenames]{color}
\usepackage{graphicx}
\usepackage[all]{xy}
\usepackage{epsfig}
\usepackage{epic}
\usepackage{eepic}
\usepackage{hyperref}
\usepackage{dsfont}\let\mathbb\mathds
\usepackage[english]{babel}

\DeclareMathAlphabet\oldmathcal{OMS}        {cmsy}{b}{n}
\SetMathAlphabet    \oldmathcal{normal}{OMS}{cmsy}{m}{n}
\DeclareMathAlphabet\oldmathbcal{OMS}       {cmsy}{b}{n}
 
\usepackage{eucal}

\usepackage[active]{srcltx}

\newtheorem{theorem}{Theorem}[section]
\newtheorem{lemma}[theorem]{Lemma}
\newtheorem{proposition}[theorem]{Proposition}
\newtheorem{corollary}[theorem]{Corollary}
\newtheorem{definition}[theorem]{Definition}

\newtheorem{def/prop}[theorem]{Definition/Proposition}

\newenvironment{example}{\medskip \refstepcounter{theorem}
\noindent  {\bf Example \thetheorem}.\rm}{\,}
\newenvironment{remark}{\medskip \refstepcounter{theorem}
\newcommand     {\comment}[1]   {}
\newcommand{\mute}[2] {}
\newcommand     {\printname}[1] {}

\noindent  {\bf Remark \thetheorem}.\rm}{\,}
\def\<{\langle}

\def\>{\rangle}

\def\BOne{{\mathchoice {\rm 1\mskip-4mu l} {\rm 1\mskip-4mu l}
                          {\rm 1\mskip-4.5mu l} {\rm 1\mskip-5mu l}}}

\def\tr{{\rm tr}~}

\def\fract#1#2{\raise4pt\hbox{$ #1 \atop #2 $}}
\def\decdnar#1{\phantom{\hbox{$\scriptstyle{#1}$}}
\left\downarrow\vbox{\vskip15pt\hbox{$\scriptstyle{#1}$}}\right.}

\def\bbc{{\mathbb C}}

\def\bbp{{\mathbb P}}
\def\bbq{{\mathbb Q}}
\def\bbr{{\mathbb R}}

\def\bbz{{\mathbb Z}}

\def\gra{\alpha}
\def\grb{\beta}

\def\grg{\gamma}

\def\grk{\kappa}
\def\grl{\lambda}

\def\gro{\omega}

\def\grr{\rho}

\def\grD{\Delta}

\def\grL{\Lambda}

\def\grS{\Sigma}

\def\bfk{{\bf k}}
\def\bfl{{\bf l}}

\def\bfv{{\bf v}}
\def\bfw{{\bf w}}

\def\cald{{\mathcal D}}

\def\calk{{\mathcal K}}

\def\calw{{\mathcal W}}

\def\la#1{\hbox to #1pc{\leftarrowfill}}
\def\ra#1{\hbox to #1pc{\rightarrowfill}}

\def\gt{{\mathfrak t}}
\def\gu{{\mathfrak u}}

\def\gA{{\mathfrak A}}

\def\hook{\mathbin{\hbox to 6pt{%
                 \vrule height0.4pt width5pt depth0pt
                 \kern-.4pt
                 \vrule height6pt width0.4pt depth0pt\hss}}}

\def\12{\xi_{k_1,k_2}}
\def\m5{M^5_{k_1,k_2}}

\begin{document}

\title{Simply Connected Manifolds with Infinitely Many Toric Contact Structures and Constant Scalar Curvature Sasaki Metrics}

\author{Charles P. Boyer and Christina W. T{\o}nnesen-Friedman}\thanks{Both authors were partially supported by grants from the Simons Foundation, CPB by (\#245002) and CWT-F by (\#208799)}
\address{Charles P. Boyer, Department of Mathematics and Statistics,
University of New Mexico, Albuquerque, NM 87131.}
\email{cboyer@math.unm.edu} 
\address{Christina W. T{\o}nnesen-Friedman, Department of Mathematics, Union
College, Schenectady, New York 12308, USA } \email{tonnesec@union.edu}



\maketitle

\markboth{Infinitely Many Toric Contact Structures}{Charles P. Boyer and Christina W. T{\o}nnesen-Friedman}


\begin{abstract}
We study a class of simply connected manifolds in all odd dimensions greater than 3 that exhibit an infinite number of toric contact structures of Reeb type that are inequivalent as contact structures. We compute the cohomology ring of our manifolds by using the join construction for Sasaki manifolds and show that all such contact structures admit a ray of compatible Sasaki metrics of constant scalar curvature (CSC). Furthermore, infinitely many such structures admit at least 3 rays of constant scalar curvature Sasaki metrics.

\end{abstract}


\section{Introduction}

The study of toric manifolds has a long and well developed history (cf. the books \cite{Oda88,Ful93,Gui94a,CoLiSc11} and references therein). These texts refer to even dimensional manifolds with the  action of a torus of one half the dimension of the manifold. The odd dimensional case has been developed only fairly recently \cite{BM93,BG00b,Ler02a,Ler04}, and there is no text devoted to this subject. This is done in the context of contact geometry. As discussed by Lerman \cite{Ler04} the topology of toric contact manifolds differs greatly from the topology of toric symplectic manifolds. For symplectic manifolds with a Hamiltonian torus action, toric manifolds are simply connected, have no torsion and the homology (cohomology) vanishes in odd degrees. In contrast none of these conditions hold in the toric contact case. However, it is true for toric contact manifolds of Reeb type that the odd dimensional cohomology vanishes through half the dimension \cite{Luo12}.

In the even dimensional case, two questions of interest in the topology of toric manifolds are: to what extent does the cohomology ring determine the toric structure of a toric manifold?, and to what extent does the cohomology ring determine the diffeomorphism (homeomorphism) type of a toric manifold? In \cite{MaSu08} a toric manifold is a smooth toric complex variety; whereas, in \cite{McD11} a toric manifold is a smooth toric symplectic manifold. In the case of both of these structures there are known rigidity theorems \cite{MaSu08,McD11}. Again the odd dimensional analogue is very different. For example, consider the Wang-Ziller manifold $M^{p,q}_{k,l}$ which is the total space of a principal $S^1$-bundle over $\bbc\bbp^p\times\bbc\bbp^q$ where $k,l$ are relatively prime integers which determine the $S^1$ action. These manifolds are toric contact manifolds, and if $p=q>1$ there are infinitely many homeomorphism types with isomorphic cohomology rings as $k$ and $l$ vary \cite{WaZi90}. Moreover, in \cite{BoTo13b} the authors gave examples of toric contact manifolds of Reeb type which admit Sasaki-Einstein metrics and that have isomorphic cohomology rings, but are not homotopy equivalent. In the present paper Theorem \ref{Kruthm} below gives necessary and sufficient conditions for homotopy equivalence in dimension 7 for the more general CSC case.

The main purpose of the present paper is to present and study a class of toric contact manifolds of Reeb type where a finiteness theorem of diffeomorphism types does hold for contact manifolds with isomorphic cohomology rings. This, however, may be the exception rather than the rule. Explicitly, we have

\begin{theorem}\label{findiff}
In each odd dimension $2p+3>5$ there exist countably infinite simply connected toric contact manifolds $M_{l_1,l_2,\bfw}$ of Reeb type depending on $4$ positive integers $l_1,l_2,w_1,w_2$ satisfying $\gcd(l_2,l_1w_i)=\gcd(w_1,w_2)=1$, and with integral cohomology ring
$$H^*(M_{l_1,l_2,\bfw},\bbz)\approx\bbz[x,y]/(w_1w_2l_1^2x^2,x^{p+1},x^2y,y^2)$$
where $x,y$ are classes of degree $2$ and $2p+1$, respectively. Furthermore, with $l_1,w_1,w_2$ fixed there are a finite number of diffeomorphism types with the given cohomology ring. Hence, in each such dimension there exist simply connected smooth manifolds with countably infinite toric contact structures of Reeb type that are inequivalent as contact structures.
\end{theorem}

The theorem applies to both the spin and non-spin case.
Theorem \ref{findiff} has the following consequences for Sasakian geometry:

\begin{theorem}\label{sasstruct}
The contact structures of Theorem \ref{findiff} admit a $p+2$ dimensional cone of Sasakian structures with a ray of constant scalar curvature Sasaki metrics and for $l_2$ large enough they admit at least $3$ such rays. Furthermore, the moduli space of Sasaki metrics with the same underlying CR structure, the $p+2$-dimensional Sasaki cone has a 2-dimensional subcone (the reduced $\bfw$-Sasaki cone, $\grk_\bfw$) that is exhausted by extremal Sasaki metrics.
\end{theorem}

As in \cite{BoTo14a} most of the CSC rays are irregular Sasakian structures. The appearance of more than one ray of CSC Sasaki metrics in the same Sasaki cone was first noticed by Legendre \cite{Leg10} in the case of $S^3$-bundles over $S^2$. As mentioned in \cite{BoTo14a} this now appears to be fairly common. 

Furthermore, in the homogeneous Wang-Ziller case when $\bfw=(1,1)$ two of the three CSC metrics are actually equivalent as discussed in Section \ref{WZsec} below. As explicit examples we consider the Wang-Ziller manifolds $M^{1,p}_{l_2,1}$ which are diffeomorphic to $S^{2p+1}$-bundles over $S^2$ of which there are precisely two.  If $p$ is odd or if both $p$ and $l_2$ are even, then $M^{1,p}_{l_2,1}$ is diffeomorphic to the trivial bundle $S^2\times S^{2p+1}$; whereas, if $p$ is even and $l_2$ is odd $M^{1,p}_{l_2,1}$ is diffeomorphic to the non-trivial $S^{2p+1}$-bundle over $S^2$, denoted $S^2\tilde{\times}S^{2p+1}$. As a consequence of the discussion above and Theorem \ref{WZcsc} we have

\begin{theorem}\label{S2S}
The manifolds $S^2\times S^{2p+1}$ and $S^2\tilde{\times}S^{2p+1}$ admit a countable infinity of toric contact structures of Reeb type that are inequivalent as contact structures and labelled by the positive integer $l_2$. Moreover, when $p\geq 5$ there are at least two rays of CSC Sasaki metrics in the reduced Sasaki cone  in each such structure. For $p=1,2,3,4$ and $l_2\geq 6,3,2,2$, respectively we also have two rays of CSC Sasaki metrics.
\end{theorem}

Notice that Theorems \ref{findiff} and \ref{sasstruct} exclude the case $p=1$, whereas, Theorem \ref{S2S} does not. The reason for this is topological. In the $p=1$ case there are two differentials at the $E_2$ level in the spectral sequence coming from the join construction, whereas, when $p>1$ there is only one. As a consequence the cohomology ring of Theorem \ref{findiff} is independent of $l_1,w_1,w_2$ when $p=1$, and the manifolds are all diffeomorphic to either $S^2\times S^{3}$ or $S^2\tilde{\times}S^{3}$ depending on the second Stiefel-Whitney class. This topological difference also occurs at the level of the K\"ahler base space of the Sasaki manifold as discussed in Section \ref{cp1cpp} below, and is related to the fact that the Hirzebruch surface phenomenon of distinct complex structures on the same 4-manifold does not directly generalize to higher dimensional projective spaces, cf. \cite{ChMaSu10}. The existence of more than one complex structure in the $p=1$ case gives rise to the bouquet phenomenon as described in \cite{Boy10b,Boy10a,Boy11,BoPa10}, and is related to the existence of different conjugacy classes of maximal tori in the contactomorphism group \cite{Ler03b}. So when $p=1$ there are toric contact structures of Reeb type that are $T^3$ equivariantly inequivalent, but $T^2$ equivariantly equivalent.

\section{$\bbc\bbp^1$-Bundles over $\bbc\bbp^p$}\label{cp1cpp}
These are examples of what have been called {\it two-stage generalized Bott towers} and their topology has been thoroughly studied in \cite{ChMaSu10,ChPaSu12}. We shall assume that $p>1$ as the $p=1$ case gives the well understood Hirzebruch surfaces. These all have the form of the projectivization of a split rank $2$ vector bundle, namely, $S_n=\bbp(\BOne\oplus L_n)$ where $\BOne$ denotes the trivial complex line bundle over $\bbc\bbp^p$ and $L_n$ denotes the complex line bundle over $\bbc\bbp^p$ whose first Chern class is $n$ times a generator of $H^2(\bbc\bbp^p,\bbz)$. These are all toric K\"ahler manifolds, and as such are classified by their Delzant polytope \cite{Del88}.

\subsection{The Diffeomorphism Type}\label{diffeosec}
In contrast to the case of Hirzebruch surfaces, when $p>1$ there is the following result of \cite{ChMaSu10,ChPaSu12}:

\begin{theorem}\label{cmsthm}
For  $p>1$any two such bundles $S_n$ and $S_{n'}$ are diffeomorphic if and only if $|n'|=|n|$, and the diffeomorphism type is completely determined by their cohomology ring. 
\end{theorem}

Furthermore, the cohomology ring of $S_n$ takes the form
\begin{equation}\label{cohring}
H^*(S_n,\bbz)=\bbz[x_1,x_2]/\bigl(x_1^{p+1},(x_2(nx_1+x_2)\bigr)
\end{equation}
where $x_1,x_2$ have degree $2$.

\subsection{Toric Structures}\label{torsec}
The following result is a consequence of a theorem of Kleinschmidt \cite{Kle88} (see Theorem 7.3.7 of \cite{CoLiSc11}) together with Theorem \ref{cmsthm}.
\begin{proposition}\label{Klprop}
When $p>1$ there is a unique projective toric structure on the smooth manifold $S_n$ for each nonnegative integer $n$.
\end{proposition}

We are also interested in toric orbifolds. Although we do not make explicit use of the combinatorial approach, it is interesting to note that toric orbifolds are all normal varieties, and normal toric varieties are classified by their {\it fans} \cite{Oda88,CoLiSc11} which are certain collections of convex rational polyhedral cones (see \cite{CoLiSc11} for precise definitions). 
In particular, orbifolds are described by {\it simplicial fans} (Theorem 3.1.19 of \cite{CoLiSc11}). However, the orbifolds of interest to us in this paper are all biholomorphic to smooth projective toric varieties; hence, their fans are that of the smooth variety. The orbifold structure is better encoded in the varieties together with line bundles or certain reflexive sheaves. These are related to polytopes.

The K\"ahler cones of interest to us are easy to describe. The K\"ahler cone $\calk(X)$ of a simplicial projective toric variety $X$ is the interior of a strongly (the origin is a face) convex polyhedral cone in $H^{1,1}(X,\bbr)$ cf. \cite{Cox97}. For an $n$-dimensional projective toric variety we also identify $\calk(X)$ with the Chow group $A_{n-1}(X)\otimes\bbr$ of Weil divisors (mod rational equivalence) on $X$. For $D\in A_{n-1}(X)$ we let $PD(D)$ denote its Poincar\'e dual, that is the corresponding class in $H^{1,1}(X,\bbr)$. In the correspondence between toric varieties and fans, Weil divisors correspond to 1-dimensional cones. Since most of the divisors of interest to us are branch divisors, following \cite{LeTo97} for labelled polytopes, we could label the 1-dimensional cones in a fan by the ramification index of the corresponding divisor. A divisor $D\in A_{n-1}(X)\otimes\bbr$ with a ramification index of $1$ is an ordinary Weil divisor, whereas a divisor with ramification index $m>1$ is a $\bbq$-divisor with coefficients in the rational numbers $\bbq$. This gives rise to {\it labelled fans}, a notion which we have not yet seen in the literature.

The orbifolds that arise in our construction take the form of log pairs $(S_n,\grD)$ where $S_n=\bbp(\BOne\oplus L_n)$ and $\grD$ is a branch divisor of the form 
$$\grD=(1-\frac{1}{m_1})D_1+ (1-\frac{1}{m_2})D_2,$$
where $D_1$ is the `zero section' of $\bbp(\BOne\oplus L_n)$ and $D_2$ is the `infinity section', and the $m_i$ are the corresponding ramification indices (see \cite{BoTo14a} for details).

\subsection{The Admissible Construction}\label{admissible}
We will now assume that $n\neq 0$ and that $(\bbc\bbp^p,\gro_{FS}$) is the standard Fubini-Study K\"ahler structure with K\"ahler metric
$g_{FS}$. In particular, $[\omega_{FS}]$ is a generator of $H^2(\bbc\bbp^p, \bbz)$. 
Then 
$(\omega_{n},g_{n}): =(2n\pi \omega_{FS}, 2n\pi g_{FS})$ satisfies that 
$( g_{n}, 
\omega_{n})$ or $(- g_{n}, 
-\omega_{n})$
is a K\"ahler structure (depending on the sign of $n$). 
In either case, we let $(\pm g_{n}, \pm \omega_{n})$ refer to the K\"ahler structure.
We write the scalar curvature of $\pm g_{n}$ as $\pm 2p s_{n}$.
[So, if e.g. $-g_{n}$ is a K\"ahler structure with positive scalar curvature, $s_{n}$ would be negative.] 
Since the (scale invariant) Ricci form of $\omega_{FS}$ is given by $s_n\omega_n$ as well as by $2\pi (p+1) \omega_{FS}$, it is easy to see that $s_n=\frac{p+1}{n}$.

Since  $L_n$ is the line bundle over $\bbc\bbp^p$ such that $c_1(L_n) = n [\omega_{FS}]$, it is clear that
$c_{1}(L_n)= [\omega_n/2\pi]$.
Then, following \cite{ACGT08}, the total space of the projectivization
$S_n=\bbp(\BOne\oplus L_n)$ is called {\it admissible}.

On these manifolds, a particular type of K\"ahler metric on $S_n$,  also called
{\it admissible}, can now be constructed \cite{ACGT08}. This construction is reviewed and adapted to the log varieties $(S_n,\grD)$ for our purposes in \cite{BoTo13b} and more generally in \cite{BoTo14a}. An admissible K\"ahler manifold is a 
special case of a K\"ahler manifold admitting a so-called Hamiltonian $2$-form \cite{ApCaGa06}. More specifically, the admissible metrics as described in  \cite{BoTo13b} and \cite{BoTo14a} admit a Hamiltonian $2$-form of order one.

In Section \ref{admissiblesummary} we will summarize the results obtained in \cite{BoTo14a} as they apply to the case at hand.

\section{The $S^3_\bfw$ Sasaki Join Construction}
We refer to \cite{BGO06,BG05} for a thorough discussion of the join construction. We begin with the standard Sasakian structure on $S^{2p+1}$ together with the Hopf fibration $S^1\ra{1.5} S^{2p+1}\fract{\pi}{\ra{1.5}} \bbc\bbp^p$, and the `weighted' Sasakian structure on $S^3$ together with the `weighted Hopf fibration' (not really a fibration) $S^1\ra{1.5} S^{3}\fract{\pi_\bfw}{\ra{1.5}} \bbc\bbp^1[\bfw]$ where $\bbc\bbp^1[\bfw]$ denotes the weighted projective space with weights $\bfw=(w_1,w_2)$. Without loss of generality we assume that $w_1\geq w_2$ and that they are relatively prime, so equality implies $\bfw=(1,1)$. 

Now the $(l_1,l_2)$-join $M_{l_1,l_2,\bfw}=S^{2p+1}\star_{l_1,l_2}S^3_\bfw$ can be defined by the commutative diagram
\begin{equation}\label{s2comdia}
\begin{matrix}  S^{2p+1}\times S^3_\bfw &&& \\
                          &\searrow\pi_L && \\
                          \decdnar{\pi\times\pi_\bfw} && S^{2p+1}\star_{l_1,l_2}S^3_\bfw &\\
                          &\swarrow\pi_1 && \\
                         \bbc\bbp^p\times\bbc\bbp^1[\bfw] &&& 
\end{matrix}
\end{equation}
where $\pi_L$ is the quotient projection by the $S^1$ action generated by the vector field 
$$L_{l_1,l_2}=\frac{1}{2l_1}\xi-\frac{1}{2l_2}\xi_\bfw.$$
Here $\xi$ is the infinitesimal generator of the standard $S^1$ action on $S^{2p+1}$ while $\xi_\bfw$ is the infinitesimal generator of the weighted $S^1$ action on $S^3$. They are also the Reeb vector fields of the natural Sasakian structures on $S^{2p+1}$ and $S^3_\bfw$, respectively. The join $S^{2p+1}\star_{l_1,l_2}S^3_\bfw$ then has an induced Sasakian structure such that $\pi_1$ is the quotient projection of the $S^1$ action generated by its Reeb vector field. Note that $l_1,l_2$ are positive integers which we assume for convenience are relatively prime. In this case we obtain an infinite sequence of simply connected $(2p+3)$-dimensional Sasakian manifolds $M_{l_1,l_2,\bfw}=S^{2p+1}\star_{l_1,l_2}S^3_\bfw$. 

\subsection{The Contact Structure and the Sasaki Cone}
The Sasakian structure on $M_{l_1,l_2,\bfw}$ is that induced by the 1-form $l_1\eta+l_2\eta_\bfw$ on $S^{2p+1}\times S^3$ where $\eta$ is the standard contact 1-form on $S^{2p+1}$ and $\eta_\bfw$ is the weighted contact 1-form on $S^3$, cf. \cite{BG05} for complete definitions and further discussion. We denote the induced contact 1-form on  $M_{l_1,l_2,\bfw}$ by $\eta_{l_1,l_2}$, and its Reeb vector field by $\xi_{l_1,l_2}$. The corresponding contact bundle is denoted by $\cald_{l_1,l_2,\bfw}=\ker\eta_{l_1,l_2}$. Note that these contact structures are all toric of Reeb type \cite{BG00b}, and they are classified by certain rational polyhedral cones \cite{Ler02a}. The Sasaki automorphism group $\gA\gu\gt(M_{l_1,l_2,\bfw})$ is the centralizer of the circle subgroup generated by $L_{l_1,l_2}$ in the product Sasaki automorphism group $U(p+1)\times T^2$ of $S^{2p+1}\times S^3_\bfw$. Thus, the connected component of the Sasaki automorphism group is $\gA\gu\gt_0(M_{l_1,l_2,\bfw})\approx U(p)\times T^2$. So a maximal torus $T^{p+2}$ has dimension $p+2$, and we denote its Lie algebra by $\gt_{p+2}$. There is an exact sequence
$$0\ra{1.8} \gt_L\ra{1.8} \gt_{p+3}\ra{1.8}\gt_{p+2}\ra{1.8} 0$$
where $\gt_L$ is the one-dimensional Lie algebra generated by $L_{l_1,l_2}$. Note that the Sasaki automorphism group of the join provides a splitting
\begin{equation}\label{gtsplit}
\gt_{p+2}=\gt_p\oplus \gt_2.
\end{equation}

The unreduced Sasaki cone $\gt^+_{p+2}$ \cite{BGS06} is the subset of $\gt_{p+2}$ that satisfies $\eta_{l_1,l_2}(X)>0$ everywhere on $M_{l_1,l_2,\bfw}$. This Sasaki cone is an invariant of the underlying contact structure and can be thought of, up to certain permutations, as the moduli space of Sasakian structures compatible with the contact structure $\cald_{l_1,l_2,\bfw}$. With the choice splitting \eqref{gtsplit} we can write the Sasaki cone of $M_{l_1,l_2,\bfw}$ as
\begin{equation}\label{gt+split}
\gt_{p+2}^+=\gt_p^{\geq 0}\oplus \gt_2^+.
\end{equation}
In this paper we are mostly concerned with the 2-dimensional sub cone $\gt^+_2$ which has been called the $\bfw$-Sasaki cone \cite{BoTo13b}, and henceforth will be denoted by $\gt^+_\bfw$. We remark that in the case that $\bfw=(1,1)$ we still have the permutation group on 2 letters, $\grS_2$, left as a symmetry. For any quasi-regular ray in $\gt^+_\bfw$ the quotient orbifold is a log pair $(S_n,\grD)$ as described in Section \ref{torsec} above (see Theorem 3.7 in \cite{BoTo14a}).

\section{The Topology of $M_{l_1,l_2,\bfw}$}
We first remark that according to Proposition 7.6.7 of \cite{BG05} $M_{l_1,l_2,\bfw}$ can be written as a bundle $S^{2p+1}\times_{S^1}L(l_2;l_1w_1,l_1w_2)$ with fiber the lens space $L(l_2;l_1w_1,l_1w_2)$ over $\bbc\bbp^p$ associated to the the Hopf fibration as a principal $S^1$ bundle.

For the homotopy groups we find from the long exact homotopy sequence of the fibration $S^1\ra{1.5} S^{2p+1}\times S^3\fract{\pi_L}{\ra{1.5}} M_{l_1,l_2,\bfw}$ we easily obtain
\begin{equation}\label{homotopygps}
\pi_1(M_{l_1,l_2,\bfw})=0,\quad \pi_2(M_{l_1,l_2,\bfw})=\bbz, \quad \pi_i(M_{l_1,l_2,\bfw})=\pi_i(S^{2p+1})\oplus \pi_i(S^3)~\text{for $i>2$}.
\end{equation}
In particular, since $p>1$ we have $\pi_3(M_{l_1,l_2,\bfw})=\bbz$ and $\pi_4(M_{l_1,l_2,\bfw})=\bbz_2$.

\subsection{The First Chern Class}
The first Chern class of the contact bundle $\cald_{l_1,l_2,\bfw}$ can easily be computed \cite{BoTo13b,BoTo14a}, viz.
\begin{equation}\label{c1}
c_1(\cald_{l_1,l_2,\bfw})=\bigl(l_2(p+1)-l_1|\bfw|\bigr)\grg,
\end{equation}
where $\grg$ is chosen to be a positive generator of $H^2(M_{l_1,l_2,\bfw},\bbz)\approx \bbz$.
We then have the topological invariant 
\begin{equation}\label{w2}
w_2(M_{l_1,l_2,\bfw})=\grr\bigl((l_2(p+1)-l_1|\bfw|)\grg\bigr)
\end{equation}
where $\grr$ is the reduction mod 2 map. So for $p$ odd $M_{l_1,l_2,\bfw}$ is spin if and only if $l_1$ is even or $w_i$ are both odd, and in this case whether $M_{l_1,l_2,\bfw}$ is spin or not is independent of $l_2$.

\subsection{The Cohomology Ring}
The cohomology ring of a toric contact manifold of Reeb type can be computed in principle from its combinatorics \cite{Luo12}. However, the cohomology ring of our join $M_{l_1,l_2,\bfw}$ was computed explicitly in \cite{BoTo13b} which dealt with special values of $l_1$ and $l_2$. The proof which we outline here for completeness holds in general. 

\begin{theorem}\label{topcpr}
If $p>1$ the join $M_{l_1,l_2,\bfw}=S^{2p+1}\star_{l_1,l_2}S^3_\bfw$ has integral cohomology ring
given by
$$H^*(M_{l_1,l_2,\bfw},\bbz)\approx\bbz[x,y]/(w_1w_2l_1^2x^2,x^{p+1},x^2y,y^2)$$
where $x,y$ are classes of degree $2$ and $2p+1$, respectively.
\end{theorem}

\begin{proof}
Our approach uses the spectral sequence for the fibration 
$$M\times S^3_\bfw \ra{2.6} M_{l_1,l_2,\bfw}\ra{2.6}
\mathsf{B}S^1$$
employed in \cite{WaZi90} and generalized to the orbifold setting in \cite{BG00a} (see also Section 7.6.2 of \cite{BG05}).
First from \cite{BoTo13b} we have
\begin{lemma}\label{cporbcoh}
For $w_1$ and $w_2$ relatively prime positive integers we have
$$H^r_{orb}(\bbc\bbp^1[\bfw],\bbz)=H^r( \mathsf{B}\bbc\bbp^1[\bfw],\bbz)= \begin{cases}
                    \bbz &\text{for $r=0,2$,}\\                  
                    \bbz_{w_1w_2} &\text{for $r>2$ even,}\\
                     0 &\text{for $r$ odd.}
                     \end{cases}$$           
\end{lemma}
Here $\mathsf{B}G$ is the classifying space of a group $G$ or Haefliger's classifying space \cite{Hae84} of an orbifold if $G$ is an orbifold.
Then we consider the commutative diagram of fibrations
\begin{equation}\label{orbifibrationexactseq}
\begin{matrix}M\times S^3_\bfw &\ra{2.6} &M_{l_1,l_2,\bfw}&\ra{2.6}
&\mathsf{B}S^1 \\
\decdnar{=}&&\decdnar{}&&\decdnar{\psi}\\
M\times S^3_\bfw&\ra{2.6} & N\times\mathsf{B}\bbc\bbp^1[\bfw]&\ra{2.6}
&\mathsf{B}S^1\times \mathsf{B}S^1\, 
\end{matrix} \qquad \qquad
\end{equation}
Now the map $\psi$ of Diagram (\ref{orbifibrationexactseq}) is that induced by the inclusion $e^{i\theta}\mapsto (e^{il_2\theta},e^{-il_1\theta})$. So noting 
$$H^*(\mathsf{B}S^1\times \mathsf{B}S^1,\bbz)=\bbz[s_1,s_2]$$ 
we see that $\psi^*s_1=l_2s$ and $\psi^*s_2=-l_1s$. This together with the fact that the only differential in Leray-Serre spectral sequence of the fibration
$$S^3_\bfw \ra{2.6} \mathsf{B}\bbc\bbp^1[\bfw]\ra{2.6} \mathsf{B}S^1$$
is $d_4:E^{0,3}_2\ra{1.6} E^{4,0}_2$ and those induced by naturality. Using Lemma \ref{cporbcoh} this gives $d_4(\gra)=w_1w_2l_1^2s^2$ in the Leray-Serre spectral sequence of the top fibration in Diagram (\ref{orbifibrationexactseq}) from which the cohomology ring follows.
\end{proof}

So an important homotopy invariant of these manifolds is the order $|H^4(M_{l_1,l_2,\bfw},\bbz)|=w_1w_2l_1^2$, and the integral first Pontrjagin class $p_1(M_{l_1,l_2,\bfw})\in H^4(M_{l_1,l_2,\bfw},\bbz)$ is a homeomorphism invariant. So Theorem \ref{topcpr} has the following
\begin{corollary}\label{ratcoh}
For $p>1$ we have
$$H^*(M_{l_1,l_2,\bfw},\bbq)\approx \bbq[x,y]/(x^2,y^2)\approx H^*(S^2\times S^{2p+1},\bbq).$$
In particular, all the rational (real) Pontrjagin classes of $M_{l_1,l_2,\bfw}$ vanish.
\end{corollary}

We are now ready for:
\begin{theorem}\label{findiffeo}
For $p\geq 1,l_1$ and $\bfw$ fixed, there are only finitely many diffeomorphism types among the manifolds $M_{l_1,l_2,\bfw}$.
\end{theorem}

\begin{proof}
Our argument is essentially that of Wang and Ziller \cite{WaZi90}.
First we note that it follows from the first statement of Corollary \ref{ratcoh} that the Sullivan minimal model \cite{Sul75,Sul77,FeOpTa08} for $M_{l_1,l_2,\bfw}$ is $(\grL(a,b),d)\otimes(\grL(c),0)$ where $db=a^2$. So $M_{l_1,l_2,\bfw}$ is formal. But then the result follows from Theorem 13.1 of \cite{Sul77} using the second statement of Corollary \ref{ratcoh}. See also \cite{KrTr91} for a simplified approach.
\end{proof}

For $p=1$ it is well known that there are precisely two diffeomorphism types, the trivial bundle $S^2\times S^3$, and the non-trivial $S^3$-bundle over $S^2$, denoted by $S^2\tilde{\times} S^3$. We have $S^2\times S^3$ if $l_1|\bfw|$ is even, and $S^2\tilde{\times} S^3$ if $l_1|\bfw|$ is odd. Furthermore, from the first Chern class \eqref{c1} we see that there are infinitely many inequivalent contact structures in this case as well. Special cases of this case are treated from somewhat different viewpoints in \cite{Boy10b} and \cite{BoPa10}.  The former deals with only regular Reeb vector fields, whereas, the latter deals with the almost regular case as well (see Section 3.4 of \cite{BoTo14a} for a definition of almost regular). The bouquet phenomenon, which occurs owing to the existence of different Hirzebruch surfaces associated with a given symplectic form, is discussed in Section 5.1 of \cite{BoPa10}. In particular, we refer to Corollaries 5.3, 5.5 and Examples 5.4 and 5.6 of that reference.

\subsection{The Homogeneous Case: Wang-Ziller Manifolds}
The homogeneous case, that is when $\bfw=(1,1)$, was studied extensively by Wang and Ziller \cite{WaZi90}. In particular, they showed that these manifolds all admit Einstein metrics. Our interest here, however, is in the Sasakian geometry of the manifolds $M_{l_1,l_2,(1,1)}$ which in the notation of \cite{WaZi90} are $M^{1,p}_{l_2,l_1}$. The intersection between the Sasaki and Einstein geometries, that is, the Sasaki-Einstein case was discussed in \cite{BG00a}. Here we consider the general homogeneous Sasakian case, that is a homogeneous manifold that is homogeneous under a Lie group that leaves the Sasakian structure invariant. As mentioned previously we assume $p>1$. We consider the product homogeneous manifold 
$$S^{2p+1}\times S^3=\frac{U(p+1)\times U(2)}{U(p)\times U(1)}$$ 
noticing that $U(p+1)$ is the Sasakian automorphism group of the standard Sasakian structure on $S^{2p+1}$. When $\bfw=(1,1)$ the $S^1$-action defined by the projection map $\pi_L$ of Diagram \eqref{s2comdia} is central in $U(p+1)\times U(2)$. Thus, for every pair of relatively prime positive integers $(l_1,l_2)$ the manifolds $M^{1,p}_{l_2,l_1}=M_{l_1,l_2,(1,1)}=S^{2p+1}\star_{l_1,l_2}S^3$ can be written as a homogeneous Sasakian manifold of the form
$$\frac{SU(p+1)\times SU(2)}{SU(p)\times U(1)}.$$

Note that in the homogeneous case $l_1$ is a homotopy invariant, and when $p>1$ the cohomology ring reduces to
$$H^*(M^{1,p}_{l_2,l_1},\bbz)\approx\bbz[x,y]/(l_1^2x^2,x^{p+1},x^2y,y^2).$$

We have the following result of Wang and Ziller:
\begin{theorem}\label{wzthm}\cite{WaZi90}
For $p>1$ and $l_1$ fixed there are only finitely many diffeomorphism types among the manifolds $M^{1,p}_{l_2,l_1}$. 
Moreover,
\begin{enumerate}
\item $M^{1,p}_{l_2,1}$ is diffeomorphic to $S^2\times S^{2p+1}$ if $p$ is odd or if $p$ is even and $l_2$ is even; whereas, if $p$ is even and $l_2$ is odd, it is diffeomorphic to the non-trivial $S^{2p+1}$-bundle over $S^2$.
\item $M^{1,p}_{l_2,2}$ is diffeomorphic to a non-trivial $\bbr\bbp^{2p+1}$-bundle over $S^2$ which for $p$ odd is independent of $l_2$, and for $p$ even depends on the class of $l_2\mod 4$.
\end{enumerate}
\end{theorem}

On the other hand as mentioned in the introduction, Wang and Ziller also prove that if $p\geq 2$ there are infinitely many homeomorphism types among the manifolds $M^{p,p}_{l_2,l_1}$ that have the same cohomology ring.  These, however, are not directly amenable to our admissible construction.

\section{The Seven Dimensional Case}
In dimension 7 there is more topological information available through the work of Kreck and Stolz \cite{KS88} among others \cite{Kru97,Kru05,Esc05}.

\subsection{The Homogeneous Case}
In the 7-dimensional case Kreck and Stolz have given a homeomorphism and diffeomorphism classification of these homogeneous manifolds \cite{KS88}. It is clear from Theorem \ref{topcpr} that $l_1$ is a homotopy invariant. Indeed, $l_1^2$ is the order of $H^4$. Furthermore, it follows from \eqref{w2} that $M_{l_1,l_2}=M^{7}_{l_1,l_2,(1,1)}$ is spin if and only if $l_2$ is even. For our case,

\begin{theorem}\label{KSthm}\cite{KS88}
The 7-manifolds $M^{1,2}_{l'_2,l_1}$ and $M^{1,2}_{l_2,l_1}$ are homeomorphic if and only if $l'_2\equiv l_2 \mod 2l_1^2$ if $l_1$ is odd or divisible by $4$. If $l_1$ is even but not divisible by $4$, then they are homeomorphic if and only if $l'_2\equiv l_2\mod l_1^2$.

$M^{1,2}_{l'_2,l_1}$ and $M^{1,2}_{l_2,l_1}$ are diffeomorphic if and only if $l'_2\equiv l_2\mod 2^{\grl_2(l_1)}7^{\grl_7(l_1)}l_1^2$, where
\begin{eqnarray*}
\grl_2(l_1)=&\begin{cases} 0 &\text{for $l_1=2,6\mod 8$} \\
                                         1 &\text{for $l_1=1,7\mod 8$} \\
                                         2 &\text{for $l_1=3,5\mod 8$} \\
                                         3 &\text{for $l_1=0,4\mod 8$} 
                                         \end{cases} \\
\grl_7(l_1)=&\begin{cases} 0 &\text{for $l_1=1,2,5,6\mod 7$} \\
                                           1 &\text{for $l_1=0,3,4\mod 7$}
                                           \end{cases}
\end{eqnarray*}
\end{theorem}

We consider some of the homeomorphism and diffeomorphism types of homogeneous Sasakian 7-manifolds.

\begin{example}\label{l1=1}
If we take $l_1=1$ we see that if $l'_2\equiv l_2\mod 2$ then $M^{1,2}_{l'_2,1}$ and $M^{1,2}_{l_2,1}$ are diffeomorphic by Theorem \ref{wzthm}. Since $c_1(\cald_{1,l_2})=(3l_2-2)\grg$, this gives countably infinite sequences of deformation classes of homogeneous Sasakian structures belonging to inequivalent contact structures on the same smooth 7-manifold, namely $S^2\times S^5$ in the spin case ($l_2$ even)  and the non-trivial $S^5$-bundle over $S^2$ in the non-spin case ($l_2$ odd). 
\end{example}

\begin{example}\label{l1=2}
When $l_1=2$ the 7-manifolds $M^{1,2}_{l'_2,2}$ and $M^{1,2}_{l_2,2}$ are diffeomorphic if and only if $l'_2\equiv l_2\mod 4$. Furthermore, they are homeomorphic if and only if they are diffeomorphic. So there are precisely four diffeomorphism (homeomorphism) types and they are all homeomorphic to a non-trivial $\bbr\bbp^5$-bundle over $S^2$. Since $c_1(\cald_{2,l_2})=(3l_2-4)\grg$, each diffeomorphism (homeomorphism) type has a countably infinite number of deformation classes of homogeneous Sasakian structures belonging to inequivalent underlying contact structures.

\end{example}

\begin{example}
Taking $l_1=5$ we see that $M^{1,2}_{l'_2,5}$ and $M^{1,2}_{l_2,5}$ are homeomorphic if and only if $l'_2\equiv l_2\mod 50$, and they are diffeomorphic if and only if $l'_2\equiv l_2\mod 100$. So there are precisely $50$ homeomorphism types and $100$ diffeomorphism types, two diffeomorphism types for each homeomorphism type. Again each diffeomorphism type has a countably infinite sequence of deformation classes of homogeneous Sasakian structures with inequivalent underlying contact structures.
\end{example}

\subsection{The Non-Homogeneous Case}
Following \cite{Kru97} we consider the homotopy type of the 7-manifolds $M_{l_1,l_2,\bfw}^7$. Now from \eqref{w2} we have $w_2(M_{l_1,l_2,\bfw}^7)\equiv (3l_2-l_1|\bfw|)\mod 2$. If the order of $H^4$ is odd, then $w_1,w_2$ and $l_1$ are all odd in which case we have $w_2(M_{l_1,l_2,\bfw}^7)\equiv l_2 \mod 2$. The first Pontrjagin class, which is a homeomorphism invariant, is given by 
\begin{equation}\label{p1}
p_1(M^7_{l_1,l_2,\bfw})=\bigl(3l_2^2-l_1^2(w_1^2+w_2^2)\bigr)x^2\in H^4(M_{l_1,l_2,\bfw}^7,\bbz)\approx \bbz_{w_1w_2l_1^2},
\end{equation} 
and  the linking form is $s(M_{l_1,l_2,\bfw}^7)=l_2^3\in (\bbz_{w_1w_2l_1^2})^*/\{\pm 1\}$ where $\bbz^*_n$ denotes the group of units in $\bbz_n$. So Theorem 5.1 of \cite{Kru97} in our case becomes

\begin{theorem}\label{Kruthm} Let $l_1,w_1,w_2$ be odd. Then $M_{l'_1,l'_2,\bfw'}^7$ and $M_{l_1,l_2,\bfw}^7$ are homotopy equivalent if and only if
\begin{enumerate}
\item $w'_1w'_2(l'_1)^2=w_1w_2l_1^2$,
\item $l'_2\equiv l_2\mod 2$,
\item $(l'_1)^2|\bfw'|^2-l_1^2|\bfw|^2\equiv 0\mod 3\in \bbz_{w_1w_2l_1^2}$,
\item $(l'_2)^3\pm l_2^3 \equiv 0 \mod w_1w_2l_1^2.$
\end{enumerate}
\end{theorem}

We are interested in the 7-manifolds with the same cohomology ring, that is, condition (1) of Theorem \ref{Kruthm} is satisfied. 
Generally, there are a countably infinite number of such 7-manifolds. If we fix the integers $l_1,w_1,w_2$ we can choose $l_2$ to be any positive integer relatively prime to $w_1,w_2$ and $l_1$. However, we can also vary $l_1,w_1,w_2$ keeping the product $w_1w_2l_1^2$ fixed. If, for example, the order of $H^4$, denoted by $|H^4|$, is a prime $p$, then there is only one possibility for a fixed $l_2$, namely, $M^7_{1,l_2,(p,1)}$, whereas, if $|H^4|=p^2$, we have two, namely, $M^7_{p,l_2,(1,1)}$ and $M^7_{1,l_2,(p^2,1)}$, etc. The latter two are not homotopy equivalent for any $l_2$ since condition (3) of Theorem \ref{Kruthm} fails. In principle one can ascertain these 7-manifolds with $l_2$ and $|H^4|$ fixed from the prime decomposition of $|H^4|$, and then check for homotopy equivalence using Theorem \ref{Kruthm}.


The homeomorphism and diffeomorphism classification was done by Escher \cite{Esc05} under the assumption that either $2,3,$ are both relatively prime to $l_1w_i$ for $i=1,2$ or both $2,3$ are not relatively prime to $l_1w_i$ for $i=1,2$. The classification is quite involved relying on further work of Kruggel \cite{Kru05} to compute the Kreck-Stolz invariants $s_1,s_2,s_3\in \bbq/\bbz$ (the invariant $s_1$ is essentially due to Eells and Kuiper \cite{EeKu62}). 

\begin{example}
Kruggel's conditions show that for each positive integer $l_2$ relatively prime to $5$ the 7-manifolds $M^7_{5,l_2,(1,1)}$ and $M^7_{1,l_2,(25,1)}$ are homotopy equivalent but they are not homeomorphic as can be seen from Equation \eqref{p1}. Notice that $29^3+21^3\equiv 0 \mod 25$, so the four non-spin 7-manifolds 
$$M^7_{5,21,(1,1)},M^7_{5,29,(1,1)},M^7_{1,21,(25,1)},M^7_{1,29,(25,1)}$$ 
are all homotopy equivalent. Notice that the first two have equal first Pontrjagin class, but are not homeomorphic by Proposition \ref{KSprop} below. The second two also have equal first Pontrjagin class, but different from the first pair. So the last two are not homeomorphic to either of the first two. However, we have not checked Escher's conditions for possible homeomorphism equivalence of the last pair. Our guess is that they are not homeomorphic.

Another example is obtained by taking $l'_2=39,l_2=89$ or vice-versa. Since $39^3-89^3\equiv 0 \mod 25$, the four non-spin 7-manifolds 
$$M^7_{5,39,(1,1)},M^7_{5,89,(1,1)},M^7_{1,39,(25,1)},M^7_{1,89,(25,1)}$$
are homotopy equivalent. By Kreck and Stolz \cite{KS88} Theorem B, we see that $M^7_{5,39,(1,1)}$ and $M^7_{5,89,(1,1)}$ are homeomorphic, but not diffeomorphic. Furthermore, neither of the last two are homeomorphic to the first pair. It is easy to obtain diffeomorphic examples from Theorem B of \cite{KS88}. Take $M^7_{5,39,(1,1)},M^7_{5,139,(1,1)}$. More generally from Theorem B we have

\begin{proposition}\label{KSprop}
Let $l'_2,l_2$ be positive integers relatively prime to $5$. Then the manifolds $M^7_{5,l_2,(1,1)}$ and $M^7_{5,l'_2,(1,1)}$ are homeomorphic if and only if $l'_2\equiv l_2 \mod 50$, and they are diffeomorphic if and only if $l'_2\equiv l_2 \mod 100$.
\end{proposition}

\end{example}

\section{Admissible extremal and CSC rays}\label{admissiblesummary}
We now give a very brief summary of the relevant results from the admissible constructions in \cite{BoTo14a}. For the full details of the construction we refer to \cite{BoTo14a}.

Consider the $\bfw$-cone $\gt^+_\bfw$ of $M_{l_1,l_2,\bfw}$, where $w_1 \geq w_2$. A quasi-regular ray in 
$\gt^+_\bfw$ is determined by a weight vector $\bfv = (v_1,v_2)$ with relative prime components. Thus the ray is determined by the ratio $b= \frac{v_2}{v_1} \in \bbq^+$. By the denseness of the quasi-regular rays in $\gt^+_\bfw$, any ray in $\gt^+_\bfw$ is determined by a choice of $b \in \bbr^+$. 

We say that the ray determined by $b \in \bbr^+$ is an "admissible extremal ray", or "admissible CSC ray", 
if, up to isotopy, the Sasaki structure given by $\bfv$ is extremal or CSC 
such that the the transverse K\"ahler structure admits a hamiltonian $2$-form of order one.
In the case where $b \in \bbq^+$, i.e. the ray is quasi-regular (or even regular), this is equivalent to
the corresponding transverse K\"ahler metric being in the same K\"ahler class as an extremal or CSC K\"ahler metric  that admits a hamiltonian $2$-form of order one, i.e., is so-called "admissible" \cite{ACGT08}. 

\begin{definition}\cite{ApCaGa06}
Let $(S,J,\gro,g)$ be a K\"ahler manifold of real dimension $2n$. On $(S,J,\gro,g)$ a {\it Hamiltonian 2-form} is a $J$-invariant 2-form $\phi$ that satisfies the differential equation
\begin{equation}\label{ham2form}
2\nabla_X\phi = d\tr\phi\wedge (JX)^\flat-d^c\tr\phi\wedge X^\flat
\end{equation}
for any vector field $X$. Here $X^\flat$ indicates the 1-form dual to $X$, and $\tr\phi$ is the trace with respect to the K\"ahler form $\gro$, i.e. $\tr\phi=g(\phi,\gro)$ where $g$ is the K\"ahler metric.
\end{definition}
"Order one'' refers to the fact that $\phi$ naturally produces one linearly independent Hamiltonian Killing field
on $(S,J,\gro,g)$. The maximal order possible of a Hamiltonian 2-form would be $n$.

We now essentially list the findings from \cite{BoTo14a} as they apply to the case in hand.
\newpage

\begin{proposition}\label{admcsc}\hfill

\medskip

\begin{itemize}
\item Any ray in  $\gt^+_\bfw$ determined by of $b \in \bbr^+\setminus\{\frac{w_2}{w_1}\}$ is
admissible extremal. This admissible extremal structure is
CSC (quasi-regular or irregular) if and only if $f(b)=0$, where 
$$
\begin{array}{ccl}
f(b) & = &  -l_1w_1^{2p+3} b^{2 p+4}\\
\\
&  + & (  l_2 + l_1  w_2)w_1^{2(p+1)} b^{2 p+3}\\
\\
& - &  (( p+1)^2 l_2 - l_1 (( p+1) w_1 + ( p+2) w_2))w_1^{p+2}  w_2^{p} b^{p + 3} \\
\\
& + &  (2  p (p+2) l_2 - (2p+3) l_1 (w_1 + w_2))w_1^{p+1}  w_2^{p+1}  b^{p + 2} \\
\\
& - &   ( (p+1)^2 l_2 - l_1 ((p+2) w_1 + ( p+1) w_2))w_1^{p}  w_2^{p+2}  b^{p + 1}\\
\\
& + &  ( l_2 + l_1  w_1)w_2^{2 (p + 1)}b\\
\\
&- &  l_1 w_2^{2p+3}.
\end{array}
$$

\item The ray in  $\gt^+_\bfw$ determined by  $b = \frac{w_2}{w_1}$
is extremal (with transverse K\"ahler structure a product). This extremal structure is
CSC (regular) if and only if $w_1=w_2=1$.

\end{itemize}
\end{proposition}

\begin{proof}(Sketch)
The first claim in the first bullet point is a straightforward consequence of Theorem 1.2 of  \cite{BoTo14a} and the proof thereof in Sections 4 and 5 of \cite{BoTo14a} which imply that the extremal metrics, constructed to prove the theorem, are in fact admissible.

For the CSC condition, we again look to Sections 4 and 5 and in particular equation (49) in \cite{BoTo14a}, setting $d_N=p$ and $A = (p+1)$.  After factoring out an inconsequential $(p+1)$ we arrive at
the stated $f(b)$.

The two statements under the second bullet are simply for completion and true from the construction of the join.
\end{proof}

\begin{remark}
If a uniqueness result, analogous to the uniqueness of smooth extremal K\"ahler metrics in the K\"ahler class \cite{ChTi05},  of extremal Sasakian metrics within its isotopy class could be established, then the statement in Proposition \ref{admcsc} could be used to decisively check for CSC rays in $\gt^+_\bfw$. So far no such uniqueness statement has been proved and thus we have to emphasize that our CSC condition for the ray is with the extra assumption that the CSC metric is the admissible extremal structure established in \cite{BoTo14a}. 
\end{remark}

\smallskip

Now, $f(b)$ in Proposition \ref{admcsc} has only $7$ terms and so by a very rough use of Descartes' rule of signs, it can have at most $6$ positive real roots counted with multiplicity. 

\subsection{Wang-Ziller manifolds}\label{WZsec}

We now assume that $w_1=w_2=1$. In this case
$$
\begin{array}{ccl}
f(b) & = &  -l_1b^{2 p+4}\\
\\
&  + & (  l_2 + l_1 ) b^{2 p+3}\\
\\
& - &  (( p+1)^2 l_2 -  (2p+3 )l_1) b^{p + 3} \\
\\
& + & 2 ( p (p+2) l_2 - (2p+3) l_1)b^{p + 2} \\
\\
& - &   ( (p+1)^2 l_2 - (2p+3 )l_1 )  b^{p + 1}\\
\\
& + &  ( l_2 + l_1  )b\\
\\
&- &  l_1.
\end{array}
$$
One can check that $f(b)$ has a root at $b=1$ of multiplicity at least $4$ and the fourth derivative $f^{(4)}(1) =  2 (1 + p)^2 (p+2) (p(p+1) l_2 -2(3+2p) l_1)$. Thus $f(b)$ has at most $3$ distinct positive real roots.
Notice that $f(b)=0 \iff f(1/b)=0$. This is to be expected. Indeed, while the rays determined by $b$ and $1/b$ are distinct in the unreduced $\gt^+_\bfw$, they represent the same ray in the reduced Sasaki cone $\gt^+/\calw$ in the case when $(w_1,w_2)=(1,1)$. The Weyl group $\calw$ of the Sasaki automorphism group is isomorphic to $\bbz_2$ and the rays determined by $(v_1,v_2)$ and
$(v_2,v_1)$ in $\gt^+_\bfw$ are equivalent.

By Decartes' sign rule, we observe that $f'(b)$ has at most 5 positive real roots counted with multiplicity and since $f'(b)$ has a triple root at $b=1$, we conclude that $f(b)$ has at most $3$ relative extrema on the interval $(0,+\infty)$.
Then, using that any roots of $f(b)$ other than $b=1$ must come in pairs of reciprocals together with that facts that
$f(0)<0$ and $\displaystyle \lim_{b\rightarrow +\infty} f(b) = -\infty$,
it is easy to see that
$f(b)$ has $3$ distinct positive real roots if and only if $f^{(4)}(1)>0$, which is equivalent to the condition 
$l_2 > \frac{2(3+2p)}{p(p+1)}l_1$.

In conclusion we have the following theorem that generalizes Theorem 1.3 of \cite{Leg10} to all dimensions.

\begin{theorem}\label{WZcsc}
The complement of the regular (CSC) ray in the sub cone $\gt^+_{(1,1)}$ of the unreduced Sasaki cone of 
$M^{1,p}_{l_2,l_1}$ is exhausted by admissible extremal rays. If $l_2 \leq \frac{2(3+2p)}{p(p+1)}l_1$, none of those admissible extremal structures are CSC. If $l_2 > \frac{2(3+2p)}{p(p+1)}l_1$, two of those admissible extremal structures are CSC. In that case, the two extra CSC rays (irregular or quasi-regular) represent the same ray in the reduced Sasaki cone and hence the reduced Sasaki cone has at least two distinct CSC rays; one regular and one irregular or quasi-regular.
\end{theorem}

Note that if the first Chern class of the contact bundle $\cald_{l_1,l_2,(1,1)}$ 
is non-positive, we have from \eqref{c1} that $l_2 \leq \frac{2}{p+1} l_1 $. Thus we easily have the following corollary.

\begin{corollary}\label{wznonmulti}
When the first Chern class of the contact bundle $\cald_{l_1,l_2,(1,1)}$ 
is non-positive,
the complement of the regular (CSC) ray in the sub cone $\gt^+_{(1,1)}$ of the unreduced Sasaki cone of 
$M^{1,p}_{l_2,l_1}$ is exhausted by admissible, non-CSC, extremal rays
\end{corollary}

Note that in the case of multiple CSC rays on the Wang-Ziller manifolds, the non-regular ray is usually irregular, but in some cases it might be quasi-regular. For $p=1$, this can already be deduced from the work in \cite{Leg10}.
For general $p$ we have for example that any co-prime $(l_1,l_2)$ satisfying that
$$2(1+2^p(2^{2+p}-(p^2+2p+5)))l_2 = (-1+2^{p+1}(2^{p+2} - (2p+3))l_1$$
satisfies that $f(b)$ has real positive rational roots 
$b=1/2, 1$, and $2$.

We refer the reader to Section 5.1 of \cite{BoTo14a} for a discussion of the Wang-Ziller manifolds in the case of $p=2$. In particular, we direct the readers attention to Theorem 5.6 of that section.

\subsection{The non-homogenous case}

We now assume that $w_1>w_2$. Note that in this case each ray in $\gt^+_\bfw$ determines a distinct ray in the reduced Sasaki cone. By assuming $w_1>w_2$, as opposed to just $w_1 \neq w_2$, we have ``used up'' the Weyl group.

As we already observed in \cite{BoTo14a}, in this case, the polynomial function $f(b)$ always has a  a root in the interval $(\frac{w_2}{w_1},+\infty)$ and thus in general we have that at least one of the admissible extremal rays in 
$\gt^+_\bfw$ is CSC. It is easy to see that
$f(b)$ has a triple root at $b=\frac{w_2}{w_1}$ (not corresponding to an admissible ray). Since $f(b)$ can have at most $6$ real positive roots, counted with multiplicity, this implies that for $w_1>w_2$, at most $3$ of the admissible extremal structures established for the rays in $\gt^+_\bfw$ are CSC.
Now, as was already observed in \cite{BoTo14a}, in the case when $w_1>w_2$, $f(b)$ does indeed have $3$ admissible CSC rays when $l_2$ is sufficiently large (see Theorem 1.3 and Section 5.2 in \cite{BoTo14a}).
Note that exactly how large $l_2$ needs to be depends on the other values $l_1,w_1$, and $w_2$. It seems to be very complicated to get some general bounds on how big $l_2$ needs to be, to guarantee the multiple admissible CSC ray phenomenon. We can however get an analogue of Corollary \ref{wznonmulti}.
\begin{proposition}
Assume $\bfw \neq (1,1)$.
When the first Chern class of the contact bundle $\cald_{l_1,l_2,\bfw}$ 
is non-positive, one and only one ray in $\gt^+_\bfw$ satisfies that the admissible extremal structure constructed in
\cite{BoTo14a} is CSC.
\end{proposition}
\begin{proof}
When
the first Chern class of the contact bundle $\cald_{l_1,l_2,\bfw}$ 
is non-positive, we have from \eqref{c1} that $l_2(p+1)\leq l_1(w_1+w_2)$. Hence we get that the coefficients in $f(b)$ of $b^{2p+4},b^{2p+3}, b^{p+3}, b^{p+2},b^{p+1}$, and $b^0$, are
{\it negative, positive, positive, negative, positive, positive}, and {\it negative} respectively. Thus, a more careful use of Descartes' rule of signs reveals that here we will have at most $4$ positive real roots of $f(b)$ counted with multiplicity. Together with the observation above that $f(b)$ has a triple root at $b=\frac{w_2}{w_1}$, this finishes the proof of the proposition.
\end{proof}

In section 5.1 of \cite{BoTo14a} we presented a few specific examples dealing with  $M_{l_1,l_2,\bfw}=S^{2p+1}\star_{l_1,l_2}S^3_\bfw$ . For instance Example 5.2 in \cite{BoTo14a} shows that while most of our admissible CSC structures are irregular, there is still an infinite number of quasi-regular CSC examples in all dimensions. Example 5.3, culminating in Proposition 5.4,  in \cite{BoTo14a} investigates carefully the case $p=2$, $l_1=1$, and $\bfw=(3,2)$ and give a precise condition on the size required of $l_2$ in order to obtain the multiple CSC rays.

We will finish this section by looking a bit closer at a $p=1$ case.

\begin{example} $N=\bbc\bbp^1$. When $p=1$, we have that
$$
\begin{array}{ccl}
f(b) & = &  -l_1w_1^{5} b^{6}\\
\\
&  + & (  l_2 + l_1  w_2)w_1^{4} b^{5}\\
\\
& - &  (4 l_2 - l_1 (2 w_1 + 3 w_2))w_1^{3}  w_2 b^{4} \\
\\
& + &  (6 l_2 - 5 l_1 (w_1 + w_2))w_1^{2}  w_2^{2}  b^{3} \\
\\
& - &   ( 4 l_2 - l_1 (3 w_1 + 2 w_2))w_1  w_2^{3}  b^{2}\\
\\
& + &  ( l_2 + l_1  w_1)w_2^{4}b\\
\\
&- &  l_1 w_2^{5}.
\end{array}
$$

Note that $f(b) = (bw_1-w_2)^3 g(b)$, where
$$
\begin{array}{ccl}
g(b) & = & -l_1w_1^2 b^3\\
\\
&+& w_1(l_2-2l_1w_2) b^2\\
\\
&-&  w_2(l_2-2l_1w_1) b\\
\\
&+& l_1 w_2^2.
\end{array}
$$
Further $g(\frac{w_2}{w_1}) = \frac{3l_1w_2^2(w_1-w_2)}{w_1} >0$ and $\displaystyle\lim_{b\rightarrow +\infty} f(b) = -\infty$, confirming the existence of at least one positive real root, not equal to the forbidden $\frac{w_2}{w_1}$.
By Descartes' sign rule, we see that for $g(b)$ to have any chance of $3$ positive real roots, we must have that $l_2>2l_1w_1$. Assuming this, then it is easy to see that $g''(b)$ is positive for $b<0$ and since $g(0) >0$ while $g'(0) <0$, there can be no negative real roots. Therefore, we can simply check the criterion for $g(b)$ to have $3$ simple real roots (which must then be positive). For this to happen, the discriminant for the cubic must be positive.
In general, the discriminant is a quartic in the variables $l_1,l_2$, so a general bound on $l_2$ in terms of $l_1$ and $\bfw = (w_1,w_2)$ is not easy to achieve. But for any given example, it is straightforward to determine.
For example, let us assume that $\bfw = (3,2)$ and $l_1=1$. Notice that then $l_1(w_1+w_2)=5$ is odd and so the diffeomorphism type of $M_{l_1,l_2,\bfw}$ is here the non-trivial $S^3$-bundle over $S^2$.
In this case, the discriminant of $g(b)$ is positive if and only if
$l_2\geq 19$. For example, when $l_2=19$ we get one rational root at $b=1/3$ and two irrational roots $b= \frac{7\pm\sqrt{37}}{3}$, corresponding to one quasi-regular CSC  ray and two irregular CSC rays.
\end{example}

\section{Iteration of the Procedure}
We can iterate the procedure with the assumption that the Sasaki manifold $M$ on the left is regular. Thus, we take $M$ to be a Wang-Ziller manifold, $M=M^{p,q}_{k_2,k_1}$ and consider the join 
$$M^{p,q}_{\bfk,\bfl,\bfw}=M^{p,q}_{k_2,k_1}\star_{l_1,l_2}S^3_\bfw.$$
It is clear that the general construction of \cite{BoTo14a} applies to this case. However, the topological question of interest is: do the manifolds $M^{p,q}_{\bfk,\bfl}$ have a finite number of diffeomorphism types when the positive integers $p,q,l_1,l_2$ are fixed, but the positive integers $k_1,k_2$ are allowed to vary? For simplicity we consider the case $p=q=1$. Let us look at the Leray-Serre spectral sequence for $M^{1,1}_{\bfk,\bfl}$ with $l_1,l_2$ positive relatively prime integers. The Wang-Ziller manifold is $M^{1,1}_{k_2,k_1}\approx S^2\times S^3$ independent of $k_1,k_2$, so the fiber in the top row of diagram \eqref{orbifibrationexactseq} is $S^2\times S^3\times S^3_\bfw$. 

We need the $d_2$ differentials in the Leray-Serre spectral sequence of the fibration 
$$S^2\times S^3\ra{2.1}\bbc\bbp^1\times \bbc\bbp^1\ra{2.1} BS^1.$$
We have $d_2:E^{0,3}_2\ra{1.4} E^{2,2}_2$ is given by $d_2(\grb)=\gra\otimes s$ where $\gra,\grb$ are the 2 and 3 classes on the fiber, and $s$ is the 2-class on the base.

Returning to the fibration in the top row of diagram \eqref{orbifibrationexactseq} we have by the commutativity of the diagram that $d_2(\grb)=l_2\gra\otimes s$. This gives $E_4^{0,3}\approx \bbz$ and $E_4^{2,2}\approx \bbz_{l_2}$. Again from the commutativity of diagram \eqref{orbifibrationexactseq} we have $d_4:E^{0,3}_4\ra{1.4} E^{4,0}_4$ is $d_4(\grg_\bfw)=w_1w_2l_1^2s^2$ which gives $E_4^{4,0}\approx \bbz_{w_1w_2l_1^2}$ and $E_\infty^{0,3}=0$. Using Poincar\'e duality and universal coefficients give the ring structure of the simply connected 7-manifold $M^{1,1}_{\bfk,\bfl,\bfw}$, viz.
$$H^*(M^{1,1}_{\bfk,\bfl,\bfw},\bbz)=\bbz[x,y,u,z]/(x^2,l_2xy,w_1w_2l_1^2y^2,z^2,u^2,zu,zx,ux,uy)$$
where $x,y$ are 2-classes, and $z,u$ are 5-classes.

Unfortunately, these appear to be non-formal by Theorem 12 of the recent work of Biswas, Fern\'andez, Mu\~noz, and Tralle \cite{BFMT14}. Their example is ours with $l_1=l_2=1$ and $\bfw=(2,1)$. Thus, it seems unlikely that a finiteness theorem of the type of Theorem \ref{findiffeo} is available for this case.

\def\cprime{$'$} \def\cprime{$'$} \def\cprime{$'$} \def\cprime{$'$}
  \def\cprime{$'$} \def\cprime{$'$} \def\cprime{$'$} \def\cprime{$'$}
  \def\cdprime{$''$} \def\cprime{$'$} \def\cprime{$'$} \def\cprime{$'$}
  \def\cprime{$'$}
\providecommand{\bysame}{\leavevmode\hbox to3em{\hrulefill}\thinspace}
\providecommand{\MR}{\relax\ifhmode\unskip\space\fi MR }
\providecommand{\MRhref}[2]{%
  \href{http://www.ams.org/mathscinet-getitem?mr=#1}{#2}
}
\providecommand{\href}[2]{#2}

\end{document}